\documentclass[12pt,a4paper]{article}
\usepackage{amsmath,amssymb,amsthm,graphicx,color}
\usepackage[left=1in,right=1in,top=1in,bottom=1in]{geometry}
\usepackage{cite}
\usepackage[british]{babel}
\usepackage{enumerate}
\usepackage{hyperref} 
\usepackage{bm}
\usepackage{comment}

\hypersetup{
    colorlinks=false,
    pdfborder={0 0 0},
}

\newtheorem{theorem}{Theorem}

\newtheorem{proposition}[theorem]{Proposition}
\newtheorem{lemma}[theorem]{Lemma}
 
\theoremstyle{definition}

\theoremstyle{remark}

\newtheorem*{remarks*}{Remarks}
\newtheorem*{remark*}{Remark}

 \allowdisplaybreaks

\newcommand{\1}[1]{{\mathbf 1}{\{#1\}}}

\newcommand{\R}{{\mathbb R}}

\newcommand{\N}{{\mathbb N}}
\newcommand{\ZP}{{\mathbb Z}_+}
\newcommand{\RP}{{\mathbb R}_+}
\newcommand{\Sp}[1]{{\mathbb S}^{#1}}

\newcommand{\X}{{\mathbb X}}

\DeclareMathOperator{\Exp}{\mathbb{E}}
\renewcommand{\Pr}{{\mathbb P}}

\DeclareMathOperator{\trace}{tr}

\newcommand{\tra}{{\scalebox{0.6}{$\top$}}}

\newcommand{\eps}{\varepsilon}

\newcommand{\ud}{{\mathrm d}}
\newcommand{\cB}{{\mathcal B}}
\newcommand{\cC}{{\mathcal C}}
\newcommand{\cD}{{\mathcal D}}

\newcommand{\cF}{{\mathcal F}}

\newcommand{\tX}{\widetilde X}

\newcommand{\as}{\ \text{a.s.}}

\newcommand{\toP}{\overset{\mathrm{p}}{\longrightarrow}}

\newcommand{\bx}{{\mathbf{x}}}
\newcommand{\by}{{\mathbf{y}}}

\newcommand{\bu}{{\mathbf{u}}}

\newcommand{\be}{{\mathbf{e}}}
\newcommand{\0}{{\mathbf{0}}}

\newcommand{\bra}{\langle}
\newcommand{\ket}{\rangle}

\newcommand{\Besq}{{\mathrm{BESQ}}}

\makeatletter
\def\namedlabel#1#2{\begingroup  
    (#2)%
    \def\@currentlabel{#2}%
    \phantomsection\label{#1}\endgroup
}
\makeatother

\begin{document}

\title{A radial invariance principle\\ for non-homogeneous random walks} 
\author{Nicholas Georgiou\thanks{Department of Mathematical Sciences, Durham University, South Road, Durham DH1 3LE, UK.}
\thanks{supported by the Heilbronn Institute for Mathematical Research.}
\and Aleksandar Mijatovi\'c\thanks{Department of Mathematics, King's College London, \& 
The Alan Turing Institute, UK. E-mail: aleks.mijatovic@gmail.com} 
\thanks{supported in part by the EPSRC grant EP/P003818/1 and the Programme on Data-Centric Engineering  funded by
Lloyd's Register Foundation.}
\and  Andrew R.\ Wade\footnotemark[1]}

\date{\today}
\maketitle

\begin{abstract}
Consider non-homogeneous zero-drift random walks in $\R^d$, $d \geq 2$,
with the asymptotic increment covariance matrix $\sigma^2 (\bu)$
satisfying 
$\bu^\tra \sigma^2 (\bu) \bu = U$ and $\trace \sigma^2 (\bu) = V$
in all 
in directions $\bu\in\Sp{d-1}$
for some positive constants $U<V$.
In this paper we establish weak 
convergence of the radial component of the walk to 
a Bessel process with dimension $V/U$. This can be viewed as an extension of 
an invariance principle of Lamperti.
\end{abstract}

\medskip


\medskip


\section{Introduction and results}
\label{sec:results}
 
A spatially homogeneous random walk on $\R^d$ whose increments have zero mean
and finite second moments is recurrent if and only if $d \leq 2$. In~\cite{gmmw}
a class of spatially non-homogeneous random walks (Markov chains) exhibiting anomalous recurrence
behaviour was described; the increments for such walks again have zero mean, but have a covariance
that depends on the current position in a certain way. In any dimension $d \geq 2$, such walks can be recurrent or transient, depending on the model parameters. 

The goal of this note is to establish an invariance principle for the radial component of the walks
studied in~\cite{gmmw}. The result can be seen as an extension of work of Lamperti~\cite{lamp2}, and is also an important ingredient in the much more involved proof of a full invariance principle that is the subject of forthcoming work. We explain these points in more detail once we have given a precise description of the model and stated the main result.

We work in $\R^d$, $d \geq 2$. 
 Write $\0$ for the origin in $\R^d$,
and let $\| \, \cdot \, \|$ denote the Euclidean norm and $\bra \,\cdot\, ,\!\, \cdot\, \ket$ the Euclidean inner product on $\R^d$.
Write $\Sp{d-1} := \{ \bu \in \R^d : \| \bu \| = 1\}$ for the unit sphere in $\R^d$.
For $\bx \in \R^d \setminus \{ \0 \}$, set $\hat \bx := \bx / \| \bx \|$.
 For definiteness, vectors $\bx \in \R^d$
are viewed as  column vectors throughout.

We now define $X=(X_n , n \in \ZP)$, 
a discrete-time, time-homogeneous Markov process on a (non-empty, unbounded) subset $\X$ of $\R^d$.
Formally, $(\X,\cB_\X)$ is a measurable space, $\X$ is a Borel
subset of $\R^d$, and $\cB_\X$ is the $\sigma$-algebra
of all $B \cap \X$ for $B$ a Borel set  in $\R^d$.
Suppose that $X_0$ is some fixed (i.e., non-random) point in $\X$. 
Write
\[ \Delta_n := X_{n+1} - X_n \]
for the increments of $X$. By assumption, given $X_0, \ldots, X_n$,
the law of $\Delta_n$ depends only on $X_n$ (and not on $n$);
so often we ease notation by taking $n=0$ and writing just $\Delta$ for $\Delta_0$.
We also use the shorthand $\Pr_\bx [ \, \cdot \, ] = \Pr [ \, \cdot \, \! \mid X_0 = \bx]$
for probabilities when the walk is started from $\bx \in \X$; similarly we use $\Exp_\bx$ for the
corresponding expectations.

We make the following moments assumption:
\begin{description}
\item[\namedlabel{ass:moments}{A0}] Suppose that 
$\sup_{\bx \in \X} \Exp_\bx [ \| \Delta \|^4  ] < \infty$.
\end{description}
The assumption \eqref{ass:moments} ensures that $\Delta$ has a well-defined mean vector $\mu(\bx) := \Exp_\bx [ \Delta ]$, and we suppose that the random walk has \emph{zero drift}:
\begin{description}
\item[\namedlabel{ass:zero_drift}{A1}] Suppose that $\mu(\bx) = \0$ for all $\bx \in \X$.
\end{description}
The assumption \eqref{ass:moments} also ensures that $\Delta$ has a well-defined covariance matrix, which we denote by
$
 M (\bx) := \Exp_\bx [ \Delta \Delta^{\!\tra} ],
$
where $\Delta$ is viewed as a column vector. 
To rule out pathological cases, we assume that $\Delta$ is \emph{uniformly non-degenerate}, in the following sense.
\begin{description}
\item[\namedlabel{ass:unif_ellip}{A2}] There exists $v > 0$ such that $\trace M(\bx) = \Exp_\bx[ \| \Delta \|^2 ] \geq v$ for all $\bx \in \X$.
\end{description}
Write $\| \, \cdot \, \|_{\rm op}$ for the matrix (operator) norm given by $\| M \|_{\rm op} = \sup_{\bu \in \Sp{d-1}} \| M \bu \|$.
The following assumption on the asymptotic stability of the covariance structure of the process  along rays is central.
\begin{description}
\item[\namedlabel{ass:cov_limit}{A3}] Suppose that there exists a positive-definite matrix function $\sigma^2$
with domain $\Sp{d-1}$ such that, as $r \to \infty$,
\[
\eps(r) := \sup_{\bx \in \X : \| \bx \| \geq r} \| M( \bx ) - \sigma^2 ( \hat\bx ) \|_{\rm op} \to 0 .
\]
\end{description}
Finally, we assume the following.
\begin{description}
\item[\namedlabel{ass:cov_form}{A4}] 
Suppose that there exist constants $U, V$ with $0 < U < V < \infty$ such that, for all
$\bu \in \Sp{d-1}$, $\bu^\tra \sigma^2 ( \bu ) \bu = U$ and $\trace \sigma^2(\bu)  = V$. In the case $2U=V$, suppose in addition that $\eps$
as defined in~\eqref{ass:cov_limit} satisfies $\eps(r) = O (r^{-\delta})$ for some $\delta>0$.
\end{description}
Informally, $V$ quantifies the total variance of the increments, while $U$ quantifies the
variance in the radial direction; necessarily $U \leq V$. The final condition in~\eqref{ass:cov_form}
is necessary to deal with the critical parameter case.
 
The main result of~\cite{gmmw} stated that under the assumptions \eqref{ass:moments}--\eqref{ass:cov_form},
we have that (i) if $2U < V$, then $\lim_{n \to \infty} \| X_n \| = +\infty$, a.s.; and (ii) if $2U \geq V$, then $\liminf_{n \to \infty} \| X_n \| \leq r_0$, a.s., for some constant $r_0 \in \RP$.

 For $n \in \ZP$ and $t \in \RP$, define
\begin{equation}
\label{eq:scaled_walk}
\tX_n (t) := n^{-1/2} X_{\lfloor nt \rfloor} .\end{equation} 
For each $n$, we view $\tX_n$ as an element of the  space
$\cD_d := \cD(\RP ; \R^d)$ of functions $f : \RP \to \R^d$ that are right-continuous
and have left limits, endowed with the Skorokhod metric: see e.g.~\cite[\S 3.5]{ek}.

\begin{theorem}
\label{thm:radial_invariance}
Suppose that \eqref{ass:moments}--\eqref{ass:cov_form} 
hold. Without loss of generality assume that $U =1$. Then $\| \tX_n \|$ converges weakly to the $V$-dimensional Bessel process
started at $0$.
\end{theorem} 

\begin{remarks*}
\label{rem:Bessel_Conv}
\begin{itemize}
\item[(i)] As $\| \tX_n \|$ is typically non-Markov,
Theorem~\ref{thm:radial_invariance} may be viewed as an extension of the invariance principle in~\cite[Thm~5.1]{lamp2},
describing the weak convergence of a sequence of non-negative Markov processes to a Bessel diffusion. 
\item[(ii)] 
It is well known that the stochastic differential equation (SDE) 
\begin{equation}
\label{eq:BesselSDE_V}
\ud \rho_t= \frac{V-1}{2\rho_t} \1{\rho_t\neq0}\ud t + \ud B_t,\quad \rho_0=x_0,
\end{equation}
satisfied by a $V$-dimensional Bessel process, 
does not possess uniqueness in law for any $V>1$ if $x_0=0$. 
Furthermore, if $V\in(1,2)$, uniqueness in law fails also in the case $x_0>0$
(see~\cite[Thm~3.2(iii)]{cherny} for both assertions). 
Hence in the proof of Theorem~\ref{thm:radial_invariance}, we work with the sequence
$\| \tX_n \|^2$ and show that it converges to the law $\Besq^V(0)$ of the squared Bessel process, 
which is uniquely determined by its SDE (see e.g.~\cite[Ch.~XI, Sec.~1]{ry}).
\item[(iii)] Theorem~\ref{thm:radial_invariance} provides a crucial step in the proof of a full invariance
principle for $\tX_n$, under additional conditions. This is the subject of forthcoming work.
Establishing a full invariance principle requires significantly more work, a large part
of which consists of characterising the limiting diffusion that can be viewed as a generalisation of the 
Bessel process to many dimensions. In the present paper this work is done for us since the limit is a (squared) Bessel
process.  
\end{itemize}
\end{remarks*}

\section{Proofs}
\label{sec:proof}

Recall that $\Delta_n :=X_{n+1}-X_n$.

\begin{lemma}
\label{lem:invariance_conditions1}
Under assumptions~\eqref{ass:moments}--\eqref{ass:cov_form}, 
for any $k\in\N$
the following limits hold:
\begin{align}
\label{eq:Lim_max_Delta}
& \lim_{n \to \infty} \frac{1}{n^\ell} \sup_{\bx\in\X} \Exp_\bx \Bigl[ \max_{0 \leq m \leq kn} \| \Delta_m \|^{2\ell} \Bigr]    =  0, \text{ for } \ell\in\{1,2\},\\
\label{eq:Lim_max_Delta_Xn}
& \lim_{n \to \infty} \frac{1}{n^2} \sup_{\bx\in\X\cap B} \Exp_\bx \Bigl[ \max_{0 \leq m \leq kn} \|\Delta_m \|^2 \|X_m \|^2\Bigr]   =  0,
\end{align}
where $B$ is any compact set in $\R^d$.
\end{lemma}

The following estimates will be useful in the proof of Lemma~\ref{lem:invariance_conditions1}. 

\begin{lemma}
\label{lem:moment_bound}
Under assumptions~\eqref{ass:moments}--\eqref{ass:cov_form},
there exists a constant $D_0 \in \RP$ such that 
\[
\Exp_\bx \Big[  \|X_m \|^\ell \Big]   \leq D_0 (m^{\ell/2} + \|\bx\|^\ell) \]
for any $\ell\in\{1,\ldots,4\}$ and all $m\in\N$, $\bx \in \X$.
\end{lemma}

\begin{proof}
First note that $\|\bx + \Delta_m\|^2 - \|\bx\|^2 = 2\langle \bx,\Delta_m\rangle + \|\Delta_m\|^2$. 
Hence by~\eqref{ass:moments} and~\eqref{ass:zero_drift},
there exists a constant 
$C_0>0$ such that 
$$\Exp[\|X_{m+1}\|^2-\|X_m\|^2 \mid X_m]= \Exp[\|\Delta_m\|^2 \mid X_m] \leq C_0 , \text{ for all } m\in\N.$$
The inequality $\Exp_\bx[\|X_{m+1}\|^2]\leq \Exp_\bx[\|X_{m}\|^2]+C_0$ follows, 
implying 
\begin{equation}
\label{eq:square_bound}
\Exp_\bx[\|X_{m}\|^2]\leq \|\bx\|^2+C_0m , \text{ for all } \bx\in\X \text{ and } m\in\N.
\end{equation}
Similarly, 
\begin{align*}
\|\bx + \Delta_m\|^4 - \|\bx\|^4 & = (\|\bx\|^2+2\langle \bx,\Delta_m\rangle + \|\Delta_m\|^2)^2 - \|\bx\|^4 \\
& \leq 6 \|\bx\|^2 \|\Delta_m\|^2+ \| \Delta_m\|^4 + 4 \|\bx\|^2 \langle \bx,\Delta_m\rangle + 4 \|\bx\| \|\Delta_m\|^3.
\end{align*}
Then by~\eqref{ass:moments} and~\eqref{ass:zero_drift} again, we get, for some  $C_1 \in \RP$,
$$
\Exp[\|X_{m+1}\|^4-\|X_m\|^4 \mid X_m]\leq  C_1 (1+\|X_m\|^2)
$$ 
for all $m\in\N$. Taking expectations and applying~\eqref{eq:square_bound}, we find
$$
\Exp_\bx [ \|X_{m+1}\|^4 ] \leq \Exp_\bx[\|X_m\|^4]+ C_2 (1+m+\|\bx\|^2),
$$
for some $C_2 \in \RP$,
which implies that, for some $C_3 \in \RP$,
\begin{align*}
\Exp_\bx[ \|X_m\|^4 ] & = \Exp_\bx [ \| X_0 \|^4 ] + \sum_{k=1}^{m-1} \left( \Exp_\bx [ \| X_{k+1} \|^4 ] - \Exp_\bx [ \| X_k \|^4 ] \right) \\
& \leq \|\bx\|^4+ C_3 (m^2+m\|\bx\|^2), \text{ for all } m\in\N \text{ and } \bx\in\X.
\end{align*}
Since 
$m\|\bx\|^2\leq  m^2+\|\bx\|^4$, the inequality in the lemma for 
$\ell=4$ follows. The case  $\ell=2$ follows from~\eqref{eq:square_bound}. The remaining cases 
are a consequence of these bounds, the Cauchy--Schwarz inequalities $\Exp_\bx\|X_m\|\leq \Exp_\bx[\|X_m\|^2]^{1/2}$ and
\[
\Exp_\bx[ \|X_m\|^3 ] \leq \Exp_\bx[\|X_m\|^4]^{1/2}  \Exp_\bx[\|X_m\|^2]^{1/2}
, \] 
and the fact that $( m^{\ell/2} + \| \bx \|^\ell )^{1/2} \leq 2^{1/2} \max ( m, \| \bx \|^2 )^{\ell/4}$.
\end{proof}

\begin{proof}[Proof of Lemma~\ref{lem:invariance_conditions1}]
Recall that $\Delta = \Delta_0$.
First we prove the statement for $\ell = 2$. Then
\[ \Exp_\bx \max_{0 \leq m \leq k n } \| \Delta_m \|^{4}
\leq \Exp_\bx  \sum_{m=0}^{kn} \| \Delta_m \|^{4} ,\]
where, by the Markov property and~\eqref{ass:moments},
\[ \Exp_\bx [ \| \Delta_m \|^{4}] = \Exp_\bx \Exp_{X_m} [ \| \Delta_m \|^{4}] \leq \sup_{\bx \in \X} \Exp_\bx [ \| \Delta \|^4 ] \leq C_1, \]
for some $C_1 < \infty$. It follows that
\[ 0 \leq \frac{1}{n^2} \Exp_\bx \max_{0 \leq m \leq k n } \| \Delta_m \|^{4} \leq \frac{C_1 (k n+1)}{n^2} \to 0 ,\]
giving the $\ell =2$ case of~\eqref{eq:Lim_max_Delta}. 
Then Lyapunov's inequality shows that
\[ \Exp_\bx \max_{0 \leq m \leq k n } \| \Delta_m \|^{2} \leq \left( \Exp_\bx \max_{0 \leq m \leq k n } \| \Delta_m \|^{4}  \right)^{1/2} ,\]
and the  $\ell =1$ case of~\eqref{eq:Lim_max_Delta} follows.

To prove~\eqref{eq:Lim_max_Delta_Xn}, take $\gamma \in(0,1/2)$ and observe that
\begin{equation}
\label{eq:Delta_m_bound}
\| \Delta_m \|^{2} \leq n^{2\gamma} +   \| \Delta_m \|^{2} \1 { \| \Delta_m \| > n^\gamma } , \text{ for all } m\in\{0,\ldots,k n \}.
\end{equation}
Hence we have from~\eqref{eq:Delta_m_bound} that
\begin{equation}
\label{eq:Delta_Xm}
\max_{0 \leq m \leq k n} \| \Delta_m \|^2  \| X_m \|^2
\leq
n^{2\gamma} \max_{0 \leq m \leq k n}   \| X_m \|^2
+
\sum_{m=0}^{k n}    \| \Delta_m \|^2 \1 { \| \Delta_m \|  > n^\gamma } \| X_m \|^2. 
\end{equation}
To bound the first term on the right-hand side of~\eqref{eq:Delta_Xm}, note that
 chain $X$
is a martingale. Hence, for any $\bx\in\X$, the non-negative process 
$\|X\|$
is a submartingale
and
Doob's $L^2$ inequality (see e.g.~\cite[Theorem 9.4]{gut}) yields 
\begin{equation}
\label{eq:Doob_ineq_X}
\Exp_\bx  \max_{0 \leq m \leq k n}   \| X_m \|^2 \leq 4 \Exp_\bx  \| X_{kn} \|^2. 
\end{equation}
For the second term on the right-hand side of~\eqref{eq:Delta_Xm}, conditioning on $X_m$ gives
\begin{align*}
\Exp_\bx \sum_{m=0}^{kn} \| \Delta_m \|^2 \1 { \| \Delta_m \|  > n^\gamma } \| X_m \|^2
& = \Exp_\bx \sum_{m=0}^{kn}   \| X_m \|^2 \Exp_{X_m}  \left[ \| \Delta_m \|^2 \1 { \| \Delta_m \|  > n^\gamma }   \right] \\
& \leq \Exp_\bx \sum_{m=0}^{kn}   \| X_m \|^2 \sup_{\by \in \X} \Exp_{\by}  \left[ \| \Delta \|^2 \1 { \| \Delta \|  > n^\gamma }   \right] ,
\end{align*}
by the Markov property. Then by~\eqref{ass:moments} we have that
\begin{align*}
\Exp_{\by}  \left[ \| \Delta \|^2 \1 { \| \Delta \|  > n^\gamma }   \right] & =
\Exp_{\by} \left[ \| \Delta \|^4 \| \Delta \|^{-2}  \1 { \| \Delta \|  > n^\gamma }   \right] \\
& \leq n^{-2\gamma} \Exp_{\by} [ \| \Delta \|^4 ] \\
& \leq C_1 n^{-2\gamma} ,
\end{align*}
for $C_1 < \infty$ and all $\by \in \X$.
It follows that 
\begin{equation}
\label{eq:sum_bound}
\Exp_\bx \sum_{m=0}^{kn}  \| \Delta_m \|^2 \1 { \| \Delta_m \|^2 > n^\gamma } \| X_m \|^2 
\leq C_1 n^{-2\gamma} \Exp_\bx \sum_{m=0}^{kn} \| X_m \|^2.
\end{equation}
The bounds in~\eqref{eq:Delta_Xm},~\eqref{eq:Doob_ineq_X} and~\eqref{eq:sum_bound},
together with 
Lemma~\ref{lem:moment_bound},
show that
$$
\Exp_\bx \max_{0 \leq m \leq kn} \| \Delta_m \|^2  \| X_m \|^2 \leq 4 D_0 n^{2\gamma} (k n+\|\bx\|^2) + 
C_1 D_0  n^{-2\gamma}(kn+1)(kn + \|\bx\|^2),
$$	
which in turn implies~\eqref{eq:Lim_max_Delta_Xn} since $\gamma \in (0,1/2)$.
\end{proof}

We need the following result from~\cite[Theorem 2.3]{gmmw}.

\begin{lemma}
Suppose that~\eqref{ass:moments}--\eqref{ass:cov_form} hold. Then the random walk is \emph{null}, i.e., for any bounded $A \subset \R^d$,  
\begin{equation}
\label{eq:null}
 \lim_{n \to \infty} \frac{1}{n} \sum_{k=0}^{n-1} \1 { X_k \in A } = 0, \as  \text{ and in } L^q \text{ for any } q \geq 1. 
\end{equation}
\end{lemma}

Write $\be_1, \ldots, \be_d$ for the standard orthonormal basis vectors
in $\R^d$. For convenience, set $\hat \0 := \be_1$.

\begin{lemma}
\label{lem:invariance_conditions2}
Suppose that~\eqref{ass:moments}--\eqref{ass:cov_form} hold
and let $k\in\N$. 
Then, for any linear functional $\phi$ on $d\times d$ matrices, i.e. $\phi:\R^{d\times d}\to \R$,
the following limits in probability hold 
\begin{equation}
\label{eq:47}
  \frac{1}{n}  \sum_{m = 0}^{kn}  \big| \phi M (X_m ) - \phi\sigma^2 ( \hat X_m ) \big| \toP 0,
\end{equation}
\begin{equation}
\label{eq:conv_prob_quad_var_rw}
  \frac{1}{n^2}  \sum_{m = 0}^{kn}  \big| \langle [M (X_m )-\sigma^2 ( \hat X_m )]X_m,X_m\rangle\big| \toP 0.
\end{equation}
\end{lemma}
\begin{proof}
Since $\phi$ is necessarily continuous (i.e. $\| \phi \|_{\rm op}<\infty$), the following estimate holds
$$\bigl|\phi M (\bx) - \phi\sigma^2 ( \hat \bx) \bigr|\leq \| \phi \|_{\rm op}\|M (\bx) - \sigma^2 ( \hat \bx)\|_{\rm op}, \text{ for any } \bx\in\R^d.$$
Hence, for any $\eps >0$,
condition~\eqref{ass:cov_limit} entails that there exists $C \in \RP$
such that 
\[
 \bigl| \phi M (X_m ) - \phi\sigma^2 ( \hat X_m ) \bigr| 
\leq \eps, \as, \text{ on } \{ \| X_m \| \geq C \} .\]
By~\eqref{ass:moments} and~\eqref{ass:cov_limit} we have $B := \sup_{\bx\in\X} \| M (\bx ) - \sigma^2( \hat \bx ) \|_{\rm op} < \infty$,
and hence
\begin{align}
\label{eq:48} 
 \frac{1}{n}  \sum_{m = 0}^{kn}  \big| \phi M (X_m ) - \phi\sigma^2 ( \hat X_m ) \big| & \leq
\frac{1}{n} \sum_{m = 0}^{kn} \eps +
\frac{1}{n}  \sum_{m = 0}^{kn} B \| \phi \|_{\rm op}\1 { \| X_m \| \leq C }    \nonumber\\
& \leq  2k \eps + \frac{B\| \phi \|_{\rm op}}{n} 
\sum_{m = 0}^{kn}  \1 { \| X_m \| \leq C } , \text{ for all } n \in \N.\end{align}
Now, by~\eqref{eq:null}, for any $C < \infty$, as $n \to \infty$,
$n^{-1} \sum_{m = 0}^{kn}  \1 { \| X_m \| \leq C } \toP 0$. 
Since $\varepsilon>0$ was arbitrary,
together with~\eqref{eq:48}, this  implies~\eqref{eq:47}.

We now establish~\eqref{eq:conv_prob_quad_var_rw}. First note that 
$$\bigl| \langle [M (\bx )-\sigma^2 ( \hat \bx )]\bx,\bx\rangle\bigr| \leq \| M (\bx ) - \sigma^2( \hat \bx ) \|_{\rm op} \|\bx\|^2,
 \text{ for any } 
\bx\in\X.
$$
Denote by $Z_n$ the random variable in~\eqref{eq:conv_prob_quad_var_rw}. By~\eqref{ass:cov_limit}, for any 
$\varepsilon>0$ there exists a constant $C < \infty$ such that 
for all $n\in\N$
we have 
\begin{align}
Z_n   & \leq
\frac{B}{n^2} 
\sum_{m = 0}^{k n} \| X_m \|^2  \1 { \| X_m \| \leq C } +
\frac{\varepsilon}{n^2} 
\sum_{m = 0}^{k n} \| X_m \|^2  \1 { \| X_m \| > C }
\nonumber\\
& \leq  2 C^2 B k /n + Z_n', \text{ where } 
Z_n':=\frac{\varepsilon}{n^2} 
\sum_{m = 0}^{k n} \| X_m \|^2,
\label{eq:quad_conv_prob} 
\end{align}
and $B$ is defined above the display in~\eqref{eq:48}.
Fix $X_0 = \bx \in \X$. Then by the $\ell =2 $ case of Lemma~\ref{lem:moment_bound},
there is a constant $D_1 < \infty$ (depending on $k$) such that
$$
\Exp_\bx \frac{1}{n^2} 
\sum_{m = 0}^{k n} \| X_m \|^2 \leq D_1, \text{ for all } n\in\N.
$$
In order to prove $Z_n\toP0$, pick arbitrary
$\varepsilon'>0$ and $\varepsilon''>0$, and set 
$\varepsilon:=\varepsilon'\varepsilon''/(4D_1)$.
Markov's inequality implies that
$$
\Pr_\bx[Z_n'>\varepsilon'/2]<\frac{2D_1}{\varepsilon'} \varepsilon < \varepsilon''
, \text{ for all } n\in\N.
$$
Pick $C < \infty$ such that the inequality in~\eqref{eq:quad_conv_prob}  holds for all $n\in\N$.
Then, for any $n\geq 4C^2Bk /\varepsilon'$, the following inequalities hold: 
$$
\Pr_\bx[Z_n>\varepsilon']\leq  \Pr_\bx[2C^2Bk /n+ Z_n'>\varepsilon']\leq \Pr_\bx[ Z_n'>\varepsilon'/2] < \varepsilon''.
$$
Since $\varepsilon''$ is arbitrary, we have that 
$\lim_{n\to\infty}\Pr_\bx[Z_n>\varepsilon']=0$ and the lemma follows.
\end{proof}

Recall that $\tX_n$ in~\eqref{eq:scaled_walk} is a continuous-time process given in terms of the scaled Markov chain $X$,
started at
$X_0=\bx\in\R^d$.
Let $Y_n:=\|\tX_n\|^2$ be the square of the radial component of 
$\tX_n$. Since the square root is continuous, the mapping theorem~\cite[Sec.~2, Thm~2.7]{bill} implies that 
Theorem~\ref{thm:radial_invariance} follows if 
we prove that $Y_n$ converges weakly to $\Besq^V(0)$ on $\cD_1$.
This fact will be established 
using~\cite[Thm~7.4.1., p.~354]{ek}.

Let 
$B_n$
denote the predictable compensator of $Y_n$.
Let $M_n:=Y_n-B_n$ be the corresponding local martingale.
Define $A_n$ as the predictable compensator of the submartingale $M_n^2$.
In particular, both $A_n$ and $B_n$ start at zero.
The following proposition establishes the conditions necessary to apply~\cite[Thm~7.4.1., p.~354]{ek}.

\begin{proposition}
\label{prop:Assumptons_radial_conv}
Suppose that~\eqref{ass:moments}--\eqref{ass:cov_form} hold, and that $U=1$.
Let $T>0$.
The following limits hold for any starting point $X_0=\bx$ in $\X$:
\begin{align}
\label{eq:Y_n_L2_bound}
\lim_{n\to\infty}\Exp_\bx \sup_{t\in[0,T]}|Y_n(t)-Y_n(t-)|^2 & = 0,\\
\label{eq:B_n_L2_bound}
\lim_{n\to\infty}\Exp_\bx \sup_{t\in[0,T]}|B_n (t) - B_n(t-)|^2 & 
=0, \\
\label{eq:lim_A_n_radial}
\lim_{n\to\infty}\Exp_\bx \sup_{t\in[0,T]}|A_n (t) - A_n(t-)|  & = 0.
\end{align}
Furthermore, under $\Pr_\bx[\cdot]$,  we have that
\begin{align}
\label{eq:B_cond}
\sup_{t\in[0,T]}\left|B_n(t)-Vt\right| & \toP 0,\\
\label{eq:A_cond}
\sup_{t\in[0,T]}\left|A_n(t)-\int_0^t4 Y_n(s) \ud s \right| & \toP 0.
\end{align}
\end{proposition}
\begin{proof}
Without loss of generality we may assume that $T=1$.
By definition, $B_n$ is a piece-wise constant right-continuous process started at zero with jumps at
 $t= k/n$, $k\in\{1,\ldots,n\}$, given by
\begin{align}
\label{eqn:compensator_of_Y}
 B_n (t) - B_n(t-) & = \frac{1}{n}  \Exp[\|X_{k}\|^2- \|X_{k-1}\|^2 \mid  X_{k-1}] \nonumber\\
& = \frac{2}{n} \Exp[  \langle X_{k-1} , \Delta_{k-1} \rangle  \mid X_{k-1}] + \frac{1}{n} \Exp[ \|\Delta_{k-1}\|^2 \mid X_{k-1}] \nonumber\\
& = \frac{1}{n} \Exp[ \|\Delta_{k-1}\|^2 \mid X_{k-1}] ,
\end{align}
using~\eqref{ass:zero_drift}, and writing $B_n(t-)=\lim_{s\uparrow t}B_n(s)$.
By~\eqref{ass:moments}, $\Exp[ \|\Delta_{k-1}\|^2 \mid X_{k-1}]$ is uniformly bounded.
Hence
$$\sup_{t\in[0,1]}|B_n (t) - B_n(t-)|^2 = \frac{1}{n^2} \max_{1 \leq k \leq n} \left| \Exp[ \|\Delta_{k-1}\|^2|X_{k-1}]\right|^2 $$
is a sequence of bounded random variables converging to zero point-wise.
Therefore the limit in~\eqref{eq:B_n_L2_bound} follows. 

Similarly, the jumps of $Y_n$ occur at times $t=k/n$ (where $k\in\{1,\ldots,n\}$) and,
writing   $Y_n(t-)=\lim_{s\uparrow t}Y_n(s)$ as usual,
can be bounded as follows:
\begin{align}
\label{Y-squared}
|Y_n(t)-Y_n(t-)|^2 & = \frac{1}{n^2} (\|X_{k}\|^2-\|X_{k-1}\|^2)^2 \nonumber\\
& \leq  \frac{1}{n^2} (\|\Delta_{k-1}\|^2+2\|X_{k-1}\|\|\Delta_{k-1}\|)^2 \nonumber\\
& \leq    \frac{2}{n^2} (\|\Delta_{k-1}\|^4+4\|X_{k-1}\|^2\|\Delta_{k-1}\|^2), 
\end{align}
using the inequality $(x+y)^2 \leq 2 (x^2 + y^2)$.
We therefore find that
\begin{align*}
\Exp_\bx \sup_{t\in[0,1]}|Y_n(t)-Y_n(t-)|^2 
& \leq \frac{2}{n^2} (\Exp_\bx \max_{1\leq k\leq n }\|\Delta_{k-1}\|^4+4\Exp_\bx \max_{1\leq k\leq n }\|X_{k-1}\|^2\|\Delta_{k-1}\|^2).
\end{align*}
Hence~\eqref{eq:Lim_max_Delta}--\eqref{eq:Lim_max_Delta_Xn} in Lemma~\ref{lem:invariance_conditions1}
imply~\eqref{eq:Y_n_L2_bound}.

The process $A_n$ is piece-wise constant and right-continuous with jumps 
$A_n(t) - A_n (t-)$ at $t=k/n$, $k\in\{1,\ldots,n\}$,
with $A_n(t-)=\lim_{s\uparrow t}A_n(s)$, satisfying
\begin{align}
\label{A-jumps} A_n (t) - A_n(t-)  & =  \Exp[M_n(t)^2-M_n(t-)^2 \mid \cF_{k-1}]  \nonumber\\
& =  \Exp[(M_n(t)-M_n(t-))^2 \mid \cF_{k-1}] \nonumber\\
& = \Exp [ ( Y_n (t) - Y_n(t-))^2 \mid \cF_{k-1}] - ( B_n (t) - B_n (t-) )^2 , \end{align}
using the fact that $B_n (t) - B_n (t-) = \Exp [  Y_n (t) - Y_n(t-) \mid \cF_{k-1} ]$,
where 
$\cF_{k-1}$ is the $\sigma$-algebra generated by $X_0, X_1,\ldots,X_{k-1}$.
Hence by~\eqref{A-jumps} with~\eqref{eqn:compensator_of_Y} and~\eqref{Y-squared}, we find that
\begin{align*}
 | A_n (t) - A_n(t-) | & \leq  \Exp[(Y_n(t)-Y_n(t-))^2 \mid X_{k-1}] + (B_n(t)-B_n(t-))^2 \\
& \leq \frac{2}{n^2} \left( \Exp [ \| \Delta_{k-1}\|^4 \mid X_{k-1} ] + 4 \| X_{k-1} \|^2 \Exp [ \| \Delta_{k-1} \|^2 \mid X_{k-1} ] \right) \\
& {} \qquad {} + \frac{1}{n^2} \Exp [ \| \Delta_{k-1} \|^2 \mid X_{k-1} ]^2 ,
\end{align*}
for $t = k/n$, $k\in\{1,\ldots,n\}$.
By~\eqref{ass:moments} we have that there exists a constant $C_1 < \infty$ such that both  $\Exp [ \| \Delta_{k-1} \|^2 \mid X_{k-1} ]$
and $\Exp [ \| \Delta_{k-1}\|^4 \mid X_{k-1} ]$ are bounded by $C_1$, a.s., so that
\[  \sup_{t \in [0,1] } | A_n (t) - A_n(t-) |  \leq \frac{2C_1 + C_1^2}{n^2}   + \frac{8C_1}{n^2} \max_{1 \leq k \leq n}  \| X_{k-1} \|^2  .\]
By Doob's $L^2$ submartingale inequality we have $\Exp_\bx[\max_{1 \leq k \leq n}\|X_{k-1}\|^2]\leq 
4\Exp_\bx\|X_{n}\|^2$, and then~\eqref{eq:lim_A_n_radial} follows from the $\ell =2$ case of  Lemma~\ref{lem:moment_bound}.

We now prove the limit in~\eqref{eq:B_cond}.
Note that~\eqref{eqn:compensator_of_Y} and the fact that $\trace M ( \bx ) = \Exp_\bx [ \| \Delta \|^2 ]$ implies that,
with the usual convention that an empty sum is zero,
\begin{equation}
\label{eq:B_def}
B_n(t) = \frac{1}{n} \sum_{m=0}^{\lfloor nt \rfloor -1} \trace M ( X_m ) .
\end{equation}
By~\eqref{ass:cov_form} it holds that 
$\trace \sigma^2(\bu)  = V$
for all $\bu \in \Sp{d-1}$. 
Hence by~\eqref{eq:B_def} we find
$$
|B_n(t)-Vt|\leq \frac{V}{n} +\frac{1}{n} \sum_{m=0}^{\lfloor nt \rfloor -1} |\trace M ( X_m ) - \trace \sigma^2(\hat X_m)|
$$
and, as trace is a linear functional on square matrices,~\eqref{eq:47} in Lemma~\ref{lem:invariance_conditions2} yields
\begin{align*}
\sup_{t\in[0,1]}\left|B_n(t)-Vt\right|
\leq \frac{V}{n} +\frac{1}{n} \sum_{m=0}^n|\trace M ( X_m ) - \trace \sigma^2(\hat X_m)|
\toP 0.
\end{align*}

Finally, we establish~\eqref{eq:A_cond}.
From~\eqref{A-jumps} with~\eqref{eqn:compensator_of_Y} and the equality in~\eqref{Y-squared}, we find that
\begin{align*}
A_n( t ) - A_n (t-) & =  \frac{1}{n^2} \Exp [ ( 2 \langle X_{k-1} , \Delta_{k-1} \rangle + \| \Delta_{k-1} \|^2 )^2 \mid X_{k-1}] - \frac{1}{n^2} \Exp[ \|\Delta_{k-1}\|^2 \mid X_{k-1}]^2  \\
& = \frac{4}{n^2} \Exp [ \langle X_{k-1} , \Delta_{k-1} \rangle^2 \mid X_{k-1} ]
+ \frac{4}{n^2} \Exp [  \langle X_{k-1} , \Delta_{k-1} \rangle \| \Delta_{k-1} \|^2 \mid X_{k-1}] \\
& {} \qquad {} 
+ \frac{1}{n^2} \Exp[ \|\Delta_{k-1}\|^4 \mid X_{k-1}] - \frac{1}{n^2} \Exp[ \|\Delta_{k-1}\|^2 \mid X_{k-1}]^2   .
\end{align*}
For any $t\in[0,1]$, denote 
$$D_n(t):=\frac{1}{n^2} \sum_{k=0}^{\lfloor nt \rfloor -1} 4\Exp[\langle X_k,\Delta_k\rangle^2 \mid X_k].$$
It follows that
\begin{align*}
A_n (t) - D_n(t)   =  \frac{1}{n^2} \sum_{k=0}^{\lfloor nt \rfloor -1} 
\left(4\Exp[\langle X_k,\Delta_k\rangle\|\Delta_k\|^2 \mid X_k] +
\Exp[\|\Delta_k\|^4 \mid X_k] - \Exp[\|\Delta_k\|^2 \mid X_k]^2\right).
\end{align*}
By~\eqref{ass:moments}, there exists a constant $C_2< \infty$ bounding uniformly all  
$\Exp[\|\Delta_k\|^\ell \mid X_k]$ for $2 \leq \ell \leq 4$ and 
 all $k\in\{0,\ldots,n\}$. Furthermore, it holds that
$$
\frac{1}{n^2} \left| \sum_{k=0}^n 
\Exp[\langle X_k,\Delta_k\rangle\|\Delta_k\|^2 \mid X_k] \right|
\leq 
\frac{1}{n^2} \sum_{k=0}^n 
\|X_k\|
\Exp[\|\Delta_k\|^3 \mid X_k].
$$
Hence, by the $\ell =1$ case of
Lemma~\ref{lem:moment_bound},
we find that
\begin{align*}
\sup_{t\in[0,1]}|A_n (t) - D_n(t)|   \leq  \frac{C_2+C_2^2}{n}+  \frac{4C_2}{n^2} \sum_{k=0}^n  \|X_k\| \toP 0.
\end{align*}
It remains to show 
$\sup_{t\in[0,1]}\bigl|D_n(t)-\int_0^t4 Y_n(s) \ud s \bigr|\toP 0$.
With this in mind, note that the following identities hold for all $k\in\{0,\ldots,n\}$:
$$
\Exp[\langle X_k,\Delta_k\rangle^2 \mid X_k] = \langle M(X_k)X_k,X_k\rangle, \text{ and }
\|X_k\|^2= \langle \sigma^2(\hat X_k)X_k,X_k\rangle;
$$
the latter is a consequence of \eqref{ass:cov_form}, which states that
$\langle \sigma^2 ( \hat X_k) \hat X_k, \hat X_k \rangle = U$, and the assumption that $U=1$. 
Since $Y_n (t) = n^{-1} \| X_{\lfloor nt \rfloor} \|^2$, we have that
\begin{align*} \int_0^t Y_n (s) \ud s & =  \sum_{k=0}^{\lfloor nt \rfloor -1} \int_{\frac{k}{n}}^{\frac{k+1}{n}} Y_n (s) \ud s + \int_{n^{-1} \lfloor nt \rfloor}^t Y_n (s) \ud s \\
& = \frac{1}{n^2} \sum_{k=0}^{\lfloor nt \rfloor -1} \| X_k \|^2 + \frac{nt - \lfloor nt \rfloor}{n^2} \| X_{\lfloor nt \rfloor} \|^2 .\end{align*}
Hence for any $t\in[0,1]$ it holds that
\begin{align}
\nonumber
\left|D_n(t)-\int_0^t4 Y_n(s) \ud s \right| & \leq \frac{4}{n^2} \|X_{\lfloor nt\rfloor}\|^2 +
 \frac{4}{n^2}  \sum_{k = 0}^{\lfloor nt \rfloor -1}  \big| \langle [M (X_k )-\sigma^2 ( \hat X_k )]X_k,X_k\rangle\big|\\
& \leq \frac{4}{n^2} \max_{0\leq k \leq n}\|X_k\|^2 +
 \frac{4}{n^2}  \sum_{k = 0}^n  \big| \langle [M (X_k )-\sigma^2 ( \hat X_k )]X_k,X_k\rangle\big|.
\label{eqn:quad_Var_prob_conv_zero}
\end{align}
Doob's $L^2$ submartingale inequality and the $\ell =2 $ case of Lemma~\ref{lem:moment_bound} 
imply that the first term on the right-hand side of~\eqref{eqn:quad_Var_prob_conv_zero}
converges to zero in $L^1$ and hence in probability. 
The second term converges to zero in probability 
by~\eqref{eq:conv_prob_quad_var_rw} in Lemma~\ref{lem:invariance_conditions2}.
\end{proof}

\begin{proof}[Proof of Theorem~\ref{thm:radial_invariance}.]
As noted in Remark~(ii) after the theorem, 
it is sufficient to prove that 
$Y_n\Rightarrow Y$, where 
$Y$ is~$\Besq^V(0)$.
Let
$g:\R\to\RP$ be given by
$g(x):=\sqrt{|x|}$
and
note that $Y$ satisfies the SDE 
$\ud Y_t = V\ud t + 2 g(Y_t)\ud B_t$, 
where $Y_0=0$.
It is easy to see that 
$|g(x)-g(y)|^2\leq |x-y|$ for all
$x,y\in\R$.
Hence pathwise uniqueness for this SDE holds for any starting point $Y_0=x_0\in\R$
by~\cite[Ch.~IX, Thm~(3.5)(ii)]{ry} (use $\rho:(0,\infty)\to(0,\infty)$, given by
$\rho(z)=4z$). Hence, by the Yamada--Watanabe theorem~\cite[Ch.~IX, Thm~(1.7)]{ry},
the uniqueness in law holds. Thus the $\cC_1$ martingale problem for $(H,\delta_{0})$
is well-posed, where $Hf:=Vf'+2g^2 f''$ for any smooth $f:\R\to\R$
and $\delta_{0}$ is the Dirac delta measure on $\R$ concentrated at zero; here $\cC_1$ denotes
the space of continuous functions from $\R$ to $\R$.
Furthermore,
any solution of this $\cC_1$ martingale problem has non-negative trajectories
because of the support of the law of $\Besq^V(0)$ (alternatively the positivity
of the paths follows from the comparison theorem~\cite[Ch.~IX, Thm~(3.7)]{ry} and the 
fact that $\Besq^0(0)$ is equal to zero at all times).
Since the drift in $H$ is constant
and $g^2$ is continuous and non-negative on $\R$, Proposition~\ref{prop:Assumptons_radial_conv}
and~\cite[Thm~7.4.1., p.~354]{ek} imply that 
$Y_n$ converges weakly to the unique solution $Y$
of the 
$\cC_1$ martingale problem 
for $(H,\delta_{0})$.
This proves Theorem~\ref{thm:radial_invariance}. 
\end{proof}

\end{document}